\documentclass[12pt,twoside]{amsart}
\usepackage{amsthm,amsfonts,amssymb,amscd,euscript}

\usepackage[matrix,arrow]{xy}

\author{Florin Ambro} 
\address{Institute of Mathematics ``Simion Stoilow'' of the Romanian
Academy\\
P.O. BOX 1-764, RO-014700 Bucharest\\ 
Romania.}
\email{florin.ambro@imar.ro}

\newcommand{\isoto}{{\overset{\sim}{\rightarrow}}}

\newcommand{\C}{{\mathbb C}}
\newcommand{\Q}{{\mathbb Q}}
\newcommand{\Z}{{\mathbb Z}}

\newcommand{\R}{{\mathbb R}}

\newcommand{\bP}{{\mathbb P}} 
\newcommand{\bA}{{\mathbb A}} 

\newcommand{\cF}{{\mathcal F}}

\newcommand{\cI}{{\mathcal I}}
\newcommand{\cK}{{\mathcal K}}

\newcommand{\cO}{{\mathcal O}}

\newcommand{\bH}{{\mathbb H}}  

\newcommand{\Exc}{\operatorname{Exc}}
\newcommand{\id}{\operatorname{id}}
\newcommand{\im}{\operatorname{Im}}

\newcommand{\Ker}{\operatorname{Ker}}
\newcommand{\LCS}{\operatorname{LCS}}

\newcommand{\mult}{\operatorname{mult}}

\newcommand{\Sing}{\operatorname{Sing}}

\newcommand{\Supp}{\operatorname{Supp}}

\theoremstyle{plain}
\newtheorem{thm}{Theorem}[section]

\newtheorem{question}[thm]{Question}
\newtheorem{lem}[thm]{Lemma}
\newtheorem{cor}[thm]{Corollary}

\theoremstyle{definition}
\newtheorem{defn}[thm]{Definition}

\newtheorem{exmp}[thm]{Example}

\newtheorem{rem}[thm]{Remark}   
\newtheorem{ack}{Acknowledgments}   

\theoremstyle{remark}

\begin{document}

\bibliographystyle{amsalpha+}
\title{An injectivity theorem}
\maketitle

\begin{abstract} 
We generalize the injectivity theorem of Esnault and Viehweg, and apply it to the structure of log canonical type divisors.
\end{abstract} 


\footnotetext[1]
{
This work was supported by a grant of the 
Romanian National Authority for Scientific Research, 
CNCS - UEFISCDI, project number PN-II-RU-TE-2011-3-0097.
}


\setcounter{section}{-1}

\section{Introduction}


We are interested in the following {\em lifting problem}: 
given a Cartier divisor $L$ on a complex variety $X$ and
a closed subvariety $Y\subset X$, when is the restriction
map
$$
\Gamma(X,\cO_X(L))\to \Gamma(Y,\cO_Y(L))
$$
surjective? The standard method is to consider the short
exact sequence 
$$
0\to \cI_Y(L)\to \cO_X(L)\to\cO_Y(L)\to 0,
$$
which induces a long exact sequence in cohomology
$$
0\to \Gamma(X,\cI_Y(L))\to \Gamma(X,\cO_X(L))\to 
\Gamma(Y,\cO_Y(L))\to H^1(X,\cI_Y(L))\stackrel{\alpha}{\to} 
H^1(X,\cO_X(L))\cdots
$$
The restriction is surjective if and only if $\alpha$ is
injective. In particular, if $H^1(X,\cI_Y(L))=0$.

If $X$ is a nonsingular proper curve, Serre 
duality answers completely the lifting problem: the restriction 
map is not surjective if and only if $L\sim K_X+Y-D$ for some
effective divisor $D$ such that $D-Y$ is not effective. 
In particular, $\deg L\le \deg(K_X+Y)$.
If $\deg L>\deg(K_X+Y)$, then $H^1(X,\cI_Y(L))=0$, and therefore
lifting holds.

If $X$ is a nonsingular projective surface, only sufficient
criteria for lifting are known (see~\cite{Zar35}).  
If $H$ is a general hyperplane section induced by a Veronese 
embedding of sufficiently large degree (depending on $L$), then 
$
\Gamma(X,\cO_X(L))\to \Gamma(H,\cO_H(L))
$
is an isomorphism (Enriques-Severi-Zariski).
If $H$ is a hyperplane section of $X$, then
$H^i(X,\cO_X(K_X+H))=0\ (i>0)$ (Picard-Severi).

These classical results were extended by Serre~\cite{Ser55} as follows: 
if $X$ is affine and $\cF$ is a quasi-coherent 
$\cO_X$-module, then $H^i(X,\cF)=0 \ (i>0)$. If $X$ is projective,
$H$ is ample and $\cF$ is a coherent $\cO_X$-module,
then $H^i(X,\cF(mH))=0 \ (i>0)$ for $m$ sufficiently large.

Kodaira~\cite{Kod53} extended Picard-Severi's result as follows: 
if $X$ is a projective complex manifold, and $H$ is an ample 
divisor, then $H^i(X,\cO_X(K_X+H))=0\ (i>0)$. This vanishing remains
true over a field of characteristic zero, but may fail in positive
characteristic (Raynaud~\cite{Ray78}). Kodaira's vanishing is central in 
the classification theory of complex algebraic varieties, but one has to 
weaken the positivity of $H$ to apply it successfully: it still holds if $H$ 
is only semiample and big (Mumford~\cite{Mum67}, Ramanujam~\cite{Ram72}), or if $K_X+H$ 
is replaced by $\lceil K_X+H\rceil$ for a $\Q$-divisor $H$ which is nef 
and big, whose fractional part is supported by a normal crossings 
divisor (Ramanujam~\cite{Ram74}, Miyaoka~\cite{Miy79}, 
Kawamata~\cite{Kaw82}, Viehweg~\cite{Vi82}). Recall that the
round up  of a real number $x$ is $\lceil x\rceil=\min\{n\in \Z;x\le n\}$, 
and the round up of a $\Q$-divisor $D=\sum_E d_E E$ is 
$\lceil D\rceil =\sum_E \lceil d_E\rceil E$.

The first lifting criterion in the absence of bigness is due to Tankeev~\cite{Tank71}: 
if $X$ is proper nonsingular and $Y\subset X$ is the general member of 
a free linear system, then the restriction 
$$
\Gamma(X,\cO_X(K_X+2Y))\to \Gamma(Y,\cO_Y(K_X+2Y))
$$ 
is surjective. Koll\'ar~\cite{Kol86} extended it to the following 
injectivity theorem: if $H$ is a semiample divisor and $D\in |m_0H|$ 
for some $m_0\ge 1$, then the homomorphism
$$
H^q(X,\cO_X(K_X+mH))\to H^q(X,\cO_X(K_X+mH+D))
$$
is injective for all $m\ge 1,q\ge 0$. Esnault and Viehweg~\cite{EV86, EVlect}
removed completely the positivity assumption, to obtain the following 
injectivity result: let $L$ be a Cartier divisor on $X$ such that 
$L\sim_\Q K_X+\sum_i b_iE_i$, where $\sum_i E_i$ is a normal crossings 
divisor and $0\le b_i\le 1$ are rational numbers. 
If $D$ is an effective divisor supported by $\sum_{0<b_i<1}E_i$, then the
homomorphism
$$
H^q(X,\cO_X(L))\to H^q(X,\cO_X(L+D))
$$
is injective, for all $q$. The original result~\cite[Theorem 5.1]{EVlect} 
was stated in terms of roots of sections of powers of line bundles, 
and restated in this logarithmic form in~\cite[Corollary 3.2]{Amb06}.
It was used in~\cite{Amb03,Amb06} to derive basic properties of
log varieties and quasi-log varieties.

The main result of this paper (Theorem~\ref{oitlog}) is that 
Esnault-Viehweg's injectivity remains true even if some components 
$E_i$ of $D$ have $b_i=1$. In fact, 
it reduces to the special case when all $b_i=1$, which has the 
following geometric interpretation:

\begin{thm} Let $X$ be a proper nonsingular variety, defined over
an algebraically closed field of characteristic zero. Let $\Sigma$
be a normal crossings divisor on $X$, let $U=X\setminus \Sigma$.
Then the restriction homomorphism 
$$
H^q(X,\cO_X(K_X+\Sigma))\to H^q(U,\cO_U(K_U))
$$ 
is injective, for all $q$.
\end{thm}

Combined with Serre vanishing on affine varieties, it gives:

\begin{cor} Let $X$ be a proper nonsingular variety, defined over
an algebraically closed field of characteristic zero. Let $\Sigma$
be a normal crossings divisor on $X$ such that $X\setminus \Sigma$ is 
contained in an affine open subset of $X$. Then 
$$
H^q(X,\cO_X(K_X+\Sigma))=0
$$ 
for $q>0$.
\end{cor}

If $X\setminus \Sigma$ itself is affine, this vanishing is due to 
Esnault and Viehweg~\cite[page 5]{EVlect}. It implies the Kodaira
vanishing theorem.

We outline the structure of this paper. After some preliminaries in
Section 1, we prove the main injectivity result in Section 2. 
The proof is similar to that of Esnault-Viehweg, except that we do not 
use duality. It is an immediate consequence of the Atiyah-Hodge Lemma and 
Deligne's degeneration of the logarithmic Hodge to de Rham spectral sequence.
In Section 3, we obtain some vanishing theorems for sheaves of logarithmic
forms of intermediate degree. The results are the same as in~\cite{EVlect}, 
except that the complement of the boundary is only contained in an affine
open subset, instead of being itself affine. They suggest that injectivity
may extend to forms of intermediate degree (Question~\ref{inde}).
In section 4, we introduce the {\em locus of totally canonical singularities}
and the {\em non-log canonical locus} of a log variety. The latter has the same 
support as the subscheme structure for the non-log canonical locus introduced 
in~\cite{Amb03}, but the scheme structure usually differ (see Remark~\ref{comp}).
In Section 5, we partially extend the injectivity theorem to the category of log varieties. 
The open subset to which we restrict is the locus of totally canonical 
singularities of some log structure. We can only prove the injectivity for the first
cohomology group.
The idea is to descend injectivity from a log resolution, and to make this work for
higher cohomology groups one needs vanishing theorems or at least the 
degeneration of the Leray spectral sequence for a certain resolution. We do not
pursue this here. In Section 6, we establish the {\em lifting property of 
$\Gamma(X,\cO_X(L))\to \Gamma(Y,\cO_Y(L))$
for a Cartier divisor $L\sim_\R K_X+B$, with $Y$ the non-log canonical locus
of $X$} (Theorem~\ref{8.2}). 
We give two applications for this unexpected property. For a 
proper generalized log Calabi-Yau variety, we show that the non-log canonical locus 
is connected and intersects every lc center (Theorem~\ref{gia}). And we obtain an extension
theorem from a union of log canonical centers, in the log canonical case (Theorem~\ref{elc}). 
We expect this extension to play a key role in the characterization of the restriction
of log canonical rings to lc centers. 
In Section 7 we list some questions that appeared naturally during this work.

\begin{ack} I would like to thank Paolo Cascini, Alexandru Dimca, J\'anos Koll\'ar, Adrian Langer,
Vyacheslav V. Shokurov and the anonymous referees for useful discussions, comments and corrections.
\end{ack}


\section{Preliminaries}



\subsection{Directed limits}


A {\em directed family} of abelian groups $(A_m)_{m\in \Z}$ consists
of homomorphisms of abelian groups $\varphi_{mn}\colon A_m\to A_n$,
for $m\le n$, such that $\varphi_{mm}=\text{id}_{A_m}$ and 
$\varphi_{np}\circ \varphi_{mn} =\varphi_{mp}$ for $m\le n\le p$.
The {\em directed limit} $\varinjlim_m A_m$ of $(A_m)_{m\in \Z}$ is 
defined as the quotient of $\oplus_{m\in \Z}A_m$ modulo the subgroup generated by 
$x_m-\varphi_{mn}(x_m)$ for all $m\le n$ and $x_m\in A_m$. The
homomorphisms $\mu_m\colon A_m\to \varinjlim_n A_n, a_m\mapsto [a_m]$
are compatible with $\varphi_{mn}$, and satisfy the following universal
property: if $B$ is an abelian group and $f_n\colon A_m\to B$ are 
homomorphisms compatible with $\varphi_{mn}$, then there exists a unique
homomorphism $f\colon \varinjlim_m A_m\to B$ such that $f_m=f\circ \mu_m$ 
for all $m$.
From the explicit description of the directed limit, the following 
properties hold: $\varinjlim_n A_n=\cup_m\mu_m(A_m)$, and 
$\Ker(A_m\to \varinjlim_n A_n)=\cup_{m\le n}\Ker(A_m\to A_n)$.
In particular, we obtain

\begin{lem}\label{dl}
Let $(A_m)_{m\in \Z}$ be a directed system of abelian groups.
\begin{itemize}
\item[1)] $A_m\to \varinjlim_n A_n$ is injective if and only if
$A_m\to A_n$ is injective for all $n\ge m$.
\item[2)] Let $(B_m)_{m\in \Z}$ be another directed family of abelian groups,
let $f_m\colon A_m\to B_m$ be a sequence of compatible homomorphisms.
They induce a homomorphism 
$
f\colon  \varinjlim_m A_m\to \varinjlim_m B_m.
$
If $f_m$ is injective for $m\ge m_0$, then $f$ is injective.
\end{itemize}
\end{lem}


\subsection{Homomorphisms induced in cohomology}

For standard notations and results, see Grothendieck~\cite[12.1.7,12.2.5]{EGA3}.
Let $f\colon X'\to X$ and $\pi\colon X\to S$ be morphisms of 
ringed spaces. Denote $\pi'=\pi\circ f\colon X'\to S$.

Let $\cF$ be an $\cO_X$-module, and $\cF'$ an $\cO_{X'}$-module.
A homomorphism of $\cO_X$-modules $u\colon \cF\to f_*\cF'$ induces functorial 
homomorphisms of $\cO_S$-modules
$$
R^qu\colon R^q\pi_*\cF\to R^q\pi'_*(\cF') \ (q\ge 0).
$$
Grothendieck-Leray constructed a spectral sequence 
$$
E_2^{pq}=R^p\pi_*(R^qf_*\cF')\Longrightarrow R^{p+q}\pi'_*(\cF').
$$

\begin{lem}\label{i1} 
The homomorphism $R^1\pi_*(f_*\cF')\to R^1\pi'_*(\cF')$, induced by 
$\id\colon f_*\cF'\to f_*\cF'$, is injective.
\end{lem}

\begin{proof} The exact sequence of terms of low degree of the Grothendieck-Leray 
spectral sequence is
$$
0\to R^1\pi_*(f_*\cF')\to R^1\pi'_*(\cF')\to \pi_*(R^1f_*\cF')
\to R^2\pi_*(f_*\cF') \to R^2\pi'_*(\cF'),
$$
and $R^1\pi_*(f_*\cF')\to R^1\pi'_*(\cF')$ is exactly the homomorphism induced
by the identity of $f_*\cF'$.
\end{proof}

The other maps $R^p\pi_*(f_*\cF')\to R^p\pi'_*(\cF')\ (p\ge 2)$, appearing in the 
spectral sequence as the edge maps $E_2^{p,0}\to H^p$,
may not be injective.

\begin{exmp} Let $f\colon X\to Y$ be the blow-up at a point of a proper smooth 
complex surface $Y$, let $E$ be the exceptional divisor. Then the map
$$
H^2(Y,f_*\cO_X(K_X+E))\to H^2(X,\cO_X(K_X+E))
$$
is not injective. In particular, the Leray spectral sequence for $f$ and $\cO_X(K_X+E)$
does not degenerate. Indeed, consider the commutative diagram
\[ 
\xymatrix{
   H^2(X,\cO_X(K_X))  \ar[r]^\gamma                      &  H^2(X,\cO_X(K_X+E))      \\
   H^2(Y,f_*\cO_X(K_X)) \ar[r]^\beta \ar[u]_\alpha  &  H^2(Y,f_*\cO_X(K_X+E)) \ar[u]^\delta   
} \]

We have $R^if_*\cO_X(K_X)=0$ for $i=1,2$. Therefore $\alpha$ is an isomorphism, from the
Leray spectral sequence.
The natural map $f_*\cO_X(K_X)\to f_*\cO_X(K_X+E)$ is an isomorphism. Therefore $\beta$ 
is an isomorphism.
By Serre duality, the dual of $\gamma$ is the inclusion $\Gamma(X,\cO_X(-E))\to \Gamma(X,\cO_X)$.
Since $X$ is proper, $\Gamma(X,\cO_X)=\C$. Therefore $\Gamma(X,\cO_X(-E))=0$. We obtain 
$\gamma^\vee=0$. Therefore $\gamma=0$.

Since $\alpha,\beta$ are isomorphisms and $\gamma=0$, we deduce $\delta=0$. But 
$H^2(Y,f_*\cO_X(K_X+E))$ is non-zero, being isomorphic to $H^2(X,\cO_X(K_X))$, which is
dual to $\Gamma(X,\cO_X)=\C$. Therefore $\delta$ is not injective.
\end{exmp}


\subsection{Weil divisors}


Let $X$ be a normal algebraic variety defined over $k$,
an algebraically closed field. A {\em prime on} $X$ is a 
reduced irreducible cycle of codimension one. 
An {\em $\R$-Weil divisor} $D$ on $X$ is a formal sum
$$
D=\sum_E d_EE,
$$
where the sum runs after all primes on $X$, and $d_E$ are 
real numbers such that $\{E;d_E\ne 0\}$ has at most finitely
many elements. It
can be viewed as an $\R$-valued function defined on all primes, 
with finite support. By restricting the 
values to $\Q$ or $\Z$, we obtain the notion of 
{\em $\Q$-Weil divisor} and {\em Weil divisor}, respectively.

Let $f\in k(X)$ be a rational function. For a prime $E$ on $X$,
let $t$ be a local parameter at the generic point of $E$.
Define $v_E(f)$ as the supremum of all $m\in \Z$ such that 
$ft^{-m}$ is regular at the generic point of $E$. If $f=0$,
then $v_E(f)=+\infty$. Else, $v_E(f)$ is a well defined
integer. We have $v_E(fg)=v_E(f)+v_E(g)$ and 
$v_E(f+g)\ge \min(v_E(f),v_E(g))$.

For non-zero $f\in k(X)$ define
$
(f)=\sum_E v_E(f)E,
$
where the sum runs after all primes on $X$. The sum has finite
support, so $(f)$ is a Weil divisor. A Weil divisor $D$ on $X$ 
is {\em linearly trivial}, denoted $D\sim 0$, if there exists 
$0\ne f\in k(X)$ such that $D=(f)$. 

\begin{defn}
Let $D$ be an $\R$-Weil divisor on $X$. We call $D$
\begin{itemize}
\item {\em $\R$-linearly trivial}, denoted $D\sim_\R 0$, if 
there exist finitely many $r_i\in \R$ and $0\ne f_i\in k(X)$ 
such that $D=\sum_i r_i(f_i)$.
\item {\em $\Q$-linearly trivial}, denoted $D\sim_\Q 0$, if 
there exist finitely many $r_i\in \Q$ and $0\ne f_i\in k(X)$ 
such that $D=\sum_i r_i(f_i)$.
\end{itemize}
\end{defn}

\begin{lem}[\cite{Sh92}, page 97]\label{1.5}
Let $E_1,\ldots,E_l$ be distinct prime divisors on $X$, and $D$ a 
$\Q$-Weil divisor on $X$. If not empty, the set
$
\{(x_1,\ldots,x_l)\in \R^l; \sum_{i=1}^l x_iE_i\sim_\R D\}
$
is an affine subspace of $\R^l$ defined over $\Q$.
\end{lem}

\begin{proof} 
Case $D=0$: the set $V_0=\{x\in \R^l;\sum_{i=1}^l x_iE_i\sim_\R 0\}$
is an $\R$-vector subspace of $\R^l$. Let $x\in V_0$. This means that there
exist finitely many non-zero rational functions $f_\alpha\in k(X)^\times$ and 
finitely many real numbers $r_\alpha\in \R$ such that 
$$
\sum_{i=1}^l x_iE_i=\sum_\alpha r_\alpha (f_\alpha).
$$
This equality of divisors is equivalent to the system of linear equations
$$
\mult_E(\sum_{i=1}^l x_iE_i)=\sum_\alpha r_\alpha \mult_E(f_\alpha),
$$
one equation for each prime divisor $E$ which may appear in the support
of $f_\alpha$, for some $\alpha$. We have $\mult_E(f_\alpha)\in \Z$. If we fix 
the $f_\alpha$, this means that $r_\alpha$ are the solutions of a linear system 
defined over $\Q$, and the corresponding $x$'s belong to an $\R$-vector subspace
of $\R^l$ defined over $\Q$. 

The above argument shows that $V_0$ is a union of vector subspaces defined over $\Q$.
Let $v_1,\ldots,v_k$ be a basis for $V_0$ over $\R$. Each $v_a$ belongs to some 
subspace of $V_0$ defined over $\Q$. That is, there exist $(w_{ab})_b$ in $V_0\cap \Q^l$
such that $v_a\in \sum_b \R w_{ab}$. It follows that the elements $w_{ab}\in V_0\cap \Q^l$
generate $V_0$ as an $\R$-vector space. Therefore $V_0$ is defined over $\Q$.

Case $D$ arbitrary: suppose $V=\{x\in \R^l;\sum_{i=1}^l x_iE_i\sim_\R D \}$ is nonempty.
Let $x\in V$. Then $\sum_{i=1}^l x_iE_i=D+\sum_\alpha r_\alpha (f_\alpha)$ for 
finitely many $r_\alpha,f_\alpha$ as above. Since $D$ has rational coefficients, the same 
argument used above shows that once $f_\alpha$ are fixed, there
exists another representation $\sum_{i=1}^l x'_iE_i=D+\sum_\alpha r'_\alpha (f_\alpha)$,
with $x'_i,r_\alpha\in \Q$. In particular, $x'\in V\cap \Q^l$. 
We have $V=x'+V_0$. Since $V_0$ is defined over $\Q$, we conclude that $V$
is an affine subspace of $\R^l$ defined over $\Q$.
\end{proof}

If $D\sim_\Q 0$, then $D$ has rational coefficients. If $D$ has rational coefficients, 
then $D\sim_\Q 0$ if and only if $D\sim_\R 0$ (by Lemma~\ref{1.5}).

Let $D$ be an $\R$-divisor on $X$. Denote 
$D^{=1}=\sum_{d_E=1} E$, $D^{\ne 1}=\sum_{d_E\ne 1} d_E E$,
$D^{<0}=\sum_{d_E<0}d_E E$, $D^{>0}=\sum_{d_E>0}d_E E$. 
The round up (down) of $D$ is defined as
$\lceil D\rceil=\sum_E \lceil d_E\rceil E$ ($\lfloor D\rfloor=\sum_E \lceil d_E\rceil E$),
where for $x\in \R$ we denote $\lfloor x\rfloor=\max\{n\in \Z;n\le x\}$ and 
$\lceil x\rceil=\min\{n\in \Z;x\le n\}$. The fractional part of $D$ is defined as 
$\{D\}=\sum_E \{d_E\}E$, where for $x\in \R$ we denote $\{ x\}=x-\lfloor x\rfloor$.

\begin{defn}
Let $D$ be an $\R$-Weil divisor on $X$. We call $D$
{\em $\R$-Cartier} ({\em $\Q$-Cartier}, {\em Cartier}) if there
exists an open covering $X=\cup_i U_i$ such that $D|_{U_i}\sim_\R 0$
($D|_{U_i}\sim_\Q 0$, $D|_{U_i}\sim 0$) for all $i$.
\end{defn}


\subsection{Complements of effective Cartier divisors}


\begin{lem}\label{ob}
Let $D$ be an effective Cartier divisor on a Noetherian scheme $X$.
Let $U=X\setminus \Supp D$ and consider the open embedding $w\colon U\subseteq X$. Then
\begin{itemize}
\item[1)] $w$ is an affine morphism.
\item[2)] Let $\cF$ be a quasi-coherent $\cO_X$-module.
The natural inclusions $\cF(mD)\subset \cF(nD)$, 
for $m\le n$, form a directed family of $\cO_X$-modules
$(\cF(mD))_{m\in \Z}$, and  
$$
\varinjlim_m\cF(mD)=w_*(\cF|_U).
$$
\item[3)] Let $\pi\colon X\to S$ be a morphism
and $\cF$ a quasi-coherent $\cO_X$-module. Then 
$
\varinjlim_mR^q\pi_*\cF(mD)\isoto R^q(\pi|_U)_*(\cF|_U)
$
for all $q$.
\end{itemize}
\end{lem}

\begin{proof}
Let $X=\cup_\alpha V_\alpha$ be an affine open covering such
$D=(f_\alpha)_\alpha$, for non-zero divisors 
$f_\alpha\in \Gamma(V_\alpha,\cO_{V_\alpha})$ such that 
$f_\alpha f_\beta^{-1}\in \Gamma(V_\alpha\cap V_\beta,\cO_X^\times)$ for all
$\alpha,\beta$.

The set $w^{-1}(V_\alpha)=U\cap V_\alpha=D(f_\alpha)$
is affine, so 1) holds. Statement 2) is local, equivalent to the known
property 
$$
\Gamma(D(f_\alpha),\cF)=
\Gamma(V_\alpha,\cF)_{f_\alpha}=\varinjlim_m\Gamma(V_\alpha,\cF(mD))=
\Gamma(V_\alpha,\varinjlim_m \cF(mD)).
$$
For 3), directed limits commute with cohomology on quasi-compact
topological spaces. Therefore 
$$
\varinjlim_mR^q\pi_*\cF(mD)\isoto R^q\pi_*(\varinjlim_m\cF(mD))=
R^q\pi_*(w_*(\cF|_U)).
$$
Since $w$ is affine, the Leray spectral sequence for $w$ degenerates to
isomorphisms 
$$
R^q\pi_*(w_*(\cF|_U))\isoto R^q(\pi|_U)_*(\cF|_U).
$$
Therefore 3) holds.
\end{proof}


\subsection{Convention on algebraic varieties}


Throughout this paper, a variety is a reduced scheme of finite type over an 
algebraically closed field $k$ of characteristic zero.


\subsection{Explicit Deligne-Du Bois complex for normal crossing varieties}

Let $X$ be a variety with at most {\em normal crossing singularities}. That is, for 
every point $P\in X$, there exist $n\ge 1$, $I\subseteq \{1,\ldots,n\}$, and an 
isomorphism of complete local $k$-algebras
$$
\frac{k[[T_1,\ldots,T_n]] }{ (\prod_{i\in I}T_i) } \isoto \widehat{\cO}_{X,P}.
$$ 
Let $\pi\colon \bar{X}\to X$ be the normalization. For $p\ge 0$, define the 
$\cO_X$-module $\tilde{\Omega}^p_{X/k}$ to be the image of the natural map
$
\Omega^p_{X/k}\to \pi_*\Omega^p_{\bar{X}/k}.
$
We have induced differentials $d\colon \tilde{\Omega}^p_{X/k}\to \tilde{\Omega}^{p+1}_{X/k}$,
and $\tilde{\Omega}^\bullet_{X/k}$ becomes a differential complex of $\cO_X$-modules.
We call the hypercohomology group
$
\bH^r(X,\tilde{\Omega}^\bullet_{X/k})
$
the $r$-th {\em de Rham cohomology group} of $X/k$, and denote it by 
$$
H^r_{DR}(X/k).
$$
If the base field is understood, we usually drop it from notation.
Let $X_\bullet$ be the simplicial algebraic variety induced by $\pi$ (see~\cite{Del74}). 
Its components 
are $X_n=(\bar{X}/X)^{\Delta_n}$, and the simplicial maps are naturally induced.
We have a natural augmentation 
$$
\epsilon\colon X_\bullet\to X.
$$
We have $X_0=\bar{X}$, $X_1=X_0\times_X X_0$, $\epsilon_0=\pi$ and 
$\delta_0,\delta_1\colon X_1\to X_0$ are the natural projections. 
For $p\ge 0$, let $\Omega^p_{X_\bullet}$ be the simplicial $\cO_{X_\bullet}$-module
with components $\Omega^p_{X_n} \ (n\ge 0)$. The $\cO_X$-module
$\epsilon_*(\Omega^p_{X_\bullet})$ is defined as the kernel of the homomorphism
$$
\delta_1^*-\delta_0^*\colon {\epsilon_0}_*\Omega^p_{X_0}\to {\epsilon_1}_*\Omega^p_{X_1}.
$$
By~\cite[Lemme 2]{DJ74}, $\epsilon$ is a smooth resolution, and 
$R^i\epsilon_*(\Omega^p_{X_\bullet})=0$ for $i>0,p\ge 0$.

\begin{lem}
For every $p$, $\tilde{\Omega}^p_X=\epsilon_*(\Omega^p_{X_\bullet})$.
\end{lem}

\begin{proof} Since $\pi\circ \delta_0=\pi\circ \delta_1$, we obtain an inclusion 
$\tilde{\Omega}^p_X\subseteq \epsilon_*(\Omega^p_{X_\bullet})$. The opposite
inclusion may be checked locally, in an \'etale neighborhood of each point. 
Therefore we may suppose 
$$
X:(\prod_{i=1}^c z_i=0)\subset \bA^{d+1}.
$$
Then $X$ has $c$ irreducible components $X_1,\ldots,X_c$, each of them 
isomorphic to $\bA^d$. The normalization $\bar{X}$ is the disjoint union of the 
$X_i$. Therefore $\Gamma(X,\epsilon_*(\Omega^p_{X_\bullet}))$ consists of 
$c$-uples $(\omega_1,\ldots,\omega_c)$ where 
$\omega_i\in \Gamma(X_i,\Omega^p_{X_i})$ satisfy the cycle condition 
$\omega_i|_{X_i\cap X_j}=\omega_j|_{X_i\cap X_j}$ for every $i<j$.

By induction on $c$, we show that $\Gamma(X,\epsilon_*(\Omega^p_{X_\bullet}))$
is the image of the homomorphism $\Gamma(\bA^{d+1},\Omega^p_{\bA^{d+1}})\to 
\Gamma(\bar{X},\Omega^p_{\bar{X}})$. The case $c=1$ is clear. Suppose $c\ge 2$.
Let $\alpha=(\omega_1,\ldots,\omega_c)$ be an element of 
$\Gamma(X,\epsilon_*(\Omega^p_{X_\bullet}))$. There exists $\omega\in 
\Gamma(\bA^{d+1},\Omega^p_{\bA^{d+1}})$ such that $\omega_c=\omega|_{X_c}$.
Then we may replace $\alpha$ by $\alpha-\omega|_X$, so that 
$$
\alpha=(\omega_1,\ldots,\omega_{c-1},0).
$$
The cycle conditions for pairs $i<c$ give $\omega_i=z_c\eta_i$, for some 
$\eta_i\in \Gamma(X_i,\Omega^p_{X_i})$. The other cycle conditions are equivalent 
to the fact that $(\eta_1,\ldots,\eta_{c-1})\in \Gamma(X',\epsilon_*(\Omega^p_{X'_\bullet}))$,
where $X':(\prod_{i=1}^{c-1} z_i=0)\subset \bA^{d+1}$. By induction, there exists 
$\eta\in \Gamma(\bA^{d+1},\Omega^p_{\bA^{d+1}})$ such that $\eta_i=\eta|_{X_i}$ for
$1\le i\le c-1$. Then $\alpha=z_c\eta|_X$.

The map $\Gamma(\bA^{d+1},\Omega^p_{\bA^{d+1}})\to \Gamma(\bar{X},\Omega^p_{\bar{X}})$ 
factors through the {\em surjection} 
$\Gamma(\bA^{d+1},\Omega^p_{\bA^{d+1}})\to \Gamma(X,\Omega^p_X)$.
Therefore its image is the same as the image of $\Gamma(X,\Omega^p_X)\to 
\Gamma(\bar{X},\Omega^p_{\bar{X}})$.
\end{proof}

It follows that $\tilde{\Omega}^\bullet_X \to R\epsilon_*(\Omega^\bullet_{X_\bullet})$ 
is a quasi-isomorphism. From~\cite{Gr66, Del74} (see~\cite[Th\'eor\`eme 4.5]{DuB81}), we deduce

\begin{thm}\label{NCHR} 
The filtered complex $(\tilde{\Omega}_X^\bullet,F)$, where $F$ is the naive filtration,
induces a spectral sequence in hypercohomology
$$
E_1^{pq}=H^q(X,\tilde{\Omega}^p_X)\Longrightarrow \bH^{p+q}(X,\tilde{\Omega}_X^\bullet)=H^{p+q}_{DR}(X/k).
$$
If $X$ is proper, this spectral sequence degenerates at $E_1$.
\end{thm}

Note $\tilde{\Omega}^0_X=\cO_X$. If $d=\dim X$, then 
$\tilde{\Omega}^d_X=\pi_*\Omega^d_{\bar{X}}$, which is a locally free $\cO_X$-module
if and only if $X$ has no singularities. If $X$ is non-singular, the natural surjections
$\Omega^p_X\to \tilde{\Omega}^p_X$ are isomorphisms, for all $p$. So our definition of 
de Rham cohomology for varieties with at most normal crossing singularities is consistent 
with Grothendieck's definition~\cite{Gr66} for nonsingular varieties.


\subsection{Differential forms with logarithmic poles}


Let $(X,\Sigma)$ be a {\em log smooth pair}, that is $X$ is a nonsingular variety and $\Sigma$ is an 
effective divisor with at most normal crossing singularities. Denote $U=X\setminus \Sigma$. Let 
$w\colon U\to X$ be the inclusion. Then $w_*(\Omega^\bullet_U)$ is the complex of rational differentials
on $X$ which are regular on $U$. We identify it with the union of $\Omega^\bullet_X\otimes \cO_X(m\Sigma)$,
after all $m\ge 0$.

Let $p\ge 0$. The {\em sheaf of germs of differential $p$-forms on $X$ with
at most logarithmic poles along $\Sigma$}, denoted $\Omega^p_X(\log\Sigma)$ (see~\cite{Del69}), 
is the sheaf whose sections on an open subset $V$ of $X$ are 
$$
\Gamma(V,\Omega^p_X(\log\Sigma))=\{\omega\in \Gamma(V,\Omega^p_X\otimes \cO_X(\Sigma));
d\omega\in \Gamma(V,\Omega^{p+1}_X\otimes \cO_X(\Sigma)) \}.
$$

It follows that $\{ \Omega^p_X(\log\Sigma),d^p\}_p$ becomes a subcomplex of $w_*(\Omega^\bullet_U)$.
It is called the {\em logarithmic de Rham complex of $(X,\Sigma)$}, denoted  by $\Omega^\bullet_X(\log\Sigma)$.

Let $n=\dim X$. Then $\Omega^p_X(\log\Sigma)=0$ if $p\notin [0,n]$. And 
$\Omega_X^n(\log\Sigma)=\Omega^p_X\otimes \cO_X(\Sigma)=\cO_X(K_X+\Sigma)$, 
where $K_X$ is the canonical divisor of $X$. 

\begin{lem}
Let $0\le p\le n$. Then $\Omega^p_X(\log\Sigma)$ is a coherent locally free extension of 
$\Omega^p_U$ to $X$. Moreover, $\Omega^0_X(\log\Sigma)=\cO_X$, 
$\wedge^p\Omega^1_X(\log\Sigma)=\Omega^p_X(\log\Sigma)$, and the wedge product induces a perfect pairing
$$
\Omega^p_X(\log\Sigma)\otimes_{\cO_X} \Omega^{n-p}_X(\log\Sigma)\to \Omega^n_X(\log \Sigma).
$$
\end{lem} 

\begin{proof}
The $\cO_X$-module $\Omega^p_X(\log\Sigma)$ is coherent, being a subsheaf of
$\Omega^p_X\otimes \cO_X(\Sigma)$. The statements 
may be checked near a fixed point, after passing to completion. Therefore it suffices
to verify the statements at the point $P=0$ for $X=\bA^n_k$ and $\Sigma=(\prod_{i\in J}z_i)$.
As in~\cite[Properties 2.2]{EVlect} for example, it can be checked that in this case 
$\Omega^p_X(\log\Sigma)_P$ is the free $\cO_{X,P}$-module with basis 
$$
\{ \frac{dz^I}{\prod_{i\in J\cap I}z_i}; I \subseteq \{1,\ldots,n\}, |I|=p\},
$$ 
where for $I=\{i_1<\cdots<i_p\}$, $dz^I$ denotes $dz_{i_1}\wedge\cdots\wedge dz_{i_p}$.
And $\prod_{i\in \emptyset}z_i=1$. All the statements follow in this case.
\end{proof}

\begin{thm} \cite{AH55,Gr66,Del69,EVlect} \label{AH}
The inclusion
$
\Omega^\bullet_X(\log \Sigma)\subset w_*(\Omega^\bullet_U)
$
is a quasi-isomorphism.
\end{thm}

\begin{proof} We claim that $\Omega^\bullet_X(\log\Sigma)\otimes \cO_X(D)$ is a subcomplex
of $w_*(\Omega^\bullet_U)$, for every divisor $D$ supported by $\Sigma$.
Indeed, the sheaves in question are locally free, so it suffices to check the statement over the open
subset $X\setminus \Sing \Sigma$, whose complement has codimension at least two in $X$.
Therefore we may suppose $\Sigma$ is non-singular. After passing to completion at a fixed point, it 
suffices to check the claim at $P=0$ for $X=\bA^1_k$ and $\Sigma=(z)$. This follows from the formula
$$
d(1\otimes z^m)=m \cdot \frac{dz}{z}  \otimes z^m \ (m\in \Z).
$$

We obtain an increasing filtration of $w_*(\Omega^\bullet_U)$ by sub-complexes
$$
\cK_m=\Omega^*_X(\log\Sigma)\otimes \cO_X(m\Sigma)\ (m\ge 0).
$$ 
We claim that the quotient complex
$
\cK_m /\cK_{m-1} 
$
is acyclic, for every $m>0$. Since $\cK_0=\Omega^\bullet_X(\log\Sigma)$ and 
$\cup_{m\ge 0}\cK_m=w_*(\Omega^\bullet_U)$, this implies that the quotient complex
$
w_*(\Omega^\bullet_U) / \Omega^\bullet_X(\log\Sigma)
$
is acyclic, or equivalently $\Omega^\bullet_X(\log \Sigma)\subset w_*(\Omega^\bullet_U)$
is a quasi-isomorphism.

To prove that $\cK_m /\cK_{m-1} \ (m>0)$ is acyclic, note that we may work locally 
near a fixed point, and we may also pass to completion (since the components of the 
two complexes are coherent). Therefore it suffices to verify the claim at $P=0$ for 
$X=\bA^n_k$ and $\Sigma=(\prod_{i\in J}z_i)$. If we denote $H_j=(z_j)$, the
claim in this case follows from the stronger statement of ~\cite[Lemma 2.10]{EVlect}:
the inclusion
$\Omega^*_X(\log\Sigma)\otimes \cO_X(D)\subset \Omega^*_X(\log\Sigma)\otimes \cO_X(D+H_j)$
is a quasi-isomorphism, for every effective divisor $D$ supported by $\Sigma$ and every
$j\in J$.
\end{proof}

\begin{thm}\cite{Del71} \label{E1D} The filtered complex 
$(\Omega^\bullet_X(\log\Sigma),F)$, where $F$ is the naive filtration, 
induces a spectral sequence in hypercohomology
$$
E_1^{pq}=H^q(X,\Omega_X^p(\log \Sigma))\Longrightarrow 
\bH^{p+q}(X,\Omega_X^\bullet(\log \Sigma)).
$$
If $X$ is proper, this spectral sequence degenerates at $E_1$.
\end{thm}

\begin{proof}
If $k=\C$, the claim follows from~\cite{Del71} and GAGA. By the Lefschetz principle, 
the claim extends to the case when $k$ is a field of characteristic zero.
\end{proof}

\begin{lem}\label{rs} For each $p\ge 0$, we have a short exact sequence
$$
0\to \cI_\Sigma\otimes \Omega_X^p(\log \Sigma)\to \Omega^p_X\to \tilde{\Omega}^p_\Sigma\to 0.
$$
\end{lem}

\begin{proof} Let $\pi\colon \bar{\Sigma}\to \Sigma$ be the normalization. We claim
that we have an exact sequence
$$
0\to \cI_\Sigma\otimes \Omega_X^p(\log \Sigma)\to \Omega^p_X\to \pi_*\Omega^p_{\bar{\Sigma}},
$$
where the second arrow is induced by the inclusion $\Omega_X^p(\log \Sigma)\subseteq \Omega^p_X\otimes \cO_X(\Sigma)$,
and the third arrow is the restriction homomorphism $\omega\mapsto \omega|_{\bar{\Sigma}}$.
Indeed, denote $\cK=\Ker(\Omega^p_X\to \pi_*\Omega^p_{\bar{\Sigma}})$.
We have to show that $\cI_\Sigma\otimes \Omega_X^p(\log \Sigma)=\cK$. 
This is a local statement which can be checked locally near each point, and since the sheaves are coherent,
we may also pass to completion. Therefore it suffices to check the equality at $P=0$ in the special case
$X=\bA^n_k$, $\Sigma=(\prod_{j\in J}z_j)$. From the explicit description of local bases for the logarithmic
sheaves, the claim holds in this case.

Finally, we compute the image of the restriction.
The restriction factors through the {\em surjection} $\Omega^p_X\to \Omega^p_\Sigma$.
Therefore the image coincides with the image of $\Omega^p_\Sigma\to \pi_*\Omega^p_{\bar{\Sigma}}$,
which by definition is $\tilde{\Omega}^p_\Sigma$.
\end{proof}


\subsection{The cyclic covering trick}

Let $X$ be an irreducible normal variety, let $T$ be a $\Q$-Weil divisor on $X$ 
such that $T\sim_\Q 0$. Let $r\ge 1$ be minimal such that $rT\sim 0$. Choose a 
rational function $\varphi\in k(X)^\times$ such that $(\varphi)=rT$. Denote by 
$$
\tau'\colon X'\to X
$$
the normalization of $X$ in the field extension $k(X)\subseteq k(X)(\sqrt[r]{\varphi})$.
The normal variety $X'$ is irreducible, since $r$ is minimal.
Choose $\psi\in k(X')^\times$ such that $\psi^r={\tau'}^*\varphi$. One computes
$$
\tau'_*\cO_{X'}=\oplus_{i=0}^{r-1}\cO_X(\lfloor iT\rfloor)\psi^i.
$$
The finite morphism $\tau'$ is Galois, with Galois group cyclic of order $r$. 
Moreover, $\tau'$ is \'etale over $X\setminus \Supp\{T\}$. 

Suppose now that $(X,\Sigma)$ is a log smooth pair structure on $X$, and the 
fractional part $\{T\}$ is supported by $\Sigma$. Then $\tau'$ is flat, $X'$ has 
at most quotient singularities (in the \'etale topology), and $X'\setminus {\tau'}^{-1}\Sigma$ 
is nonsingular. Let $\mu\colon Y\to X'$ be an embedded resolution of singularities of 
$(X',{\tau'}^{-1}\Sigma)$. If we denote $\tau=\tau'\circ \mu$, then 
$\tau^{-1}(\Sigma)=\Sigma_Y$ is a normal crossings divisor and 
$\mu\colon Y\setminus \Sigma_Y\to X'\setminus {\tau'}^{-1}\Sigma$ 
is an isomorphism. We obtain a commutative diagram
\[ 
\xymatrix{
X' \ar[d]_{\tau'} & Y\ar[l]_\mu \ar[dl]^\tau  \\ 
X
} \]

\begin{lem}\label{keyct}\cite{RC81,Amb13}
$R^q\tau_*\Omega^p_Y(\log \Sigma_Y)=0$ for $q>0$, and 
\begin{align*}
\tau_*\Omega^p_Y(\log \Sigma_Y) & =
\Omega^p_X(\log \Sigma_X)\otimes \tau'_*\cO_{X'}        \\
 & \simeq \oplus_{i=0}^{r-1}\Omega^p_X(\log \Sigma_X)
\otimes \cO_X(\lfloor iT\rfloor).
\end{align*}
\end{lem}

This statement is proved in~\cite[Lemme 1.2, 1.3]{RC81} with two extra assumptions:
$X$ is projective, and $\Sigma$ is a {\em simple} normal crossing divisor, that is
it has normal crossing singularities and its irreducible components are smooth.
One can show that the projectivity assumption is not necessary, and the 
normal crossings case reduces to the simple normal crossing case, by 
\'etale base change (see~\cite{Amb13}).


\section{Injectivity for open embeddings}


Let $(X,\Sigma)$ be a log smooth pair, with $X$ proper. Denote $U=X\setminus \Sigma$.

\begin{thm}\label{oi} The restriction homomorphism 
$H^q(X,\cO_X(K_X+\Sigma))\to H^q(U,\cO_U(K_U))$ is injective,
for all $q$.
\end{thm}

\begin{proof} Consider the inclusion of filtered differential complexes of $\cO_X$-modules
$$
(\Omega_X^\bullet(\log \Sigma),F)\subset (w_*(\Omega^\bullet_U),F),
$$
where $F$ is the naive filtration of a complex. Let $n=\dim X$.
The inclusion $F^n\subseteq F^0$ induces a commutative diagram
\[ 
\xymatrix{
\bH^{q+n}(X,F^n\Omega^\bullet_X(\log \Sigma)) \ar[r]^\beta\ar[d]_{\alpha^n} & 
\bH^{q+n}(X,\Omega^\bullet_X(\log \Sigma)) \ar[d]^\alpha \\
\bH^{q+n}(X,F^nw_*(\Omega^\bullet_U))  \ar[r] &  
\bH^{q+n}(X,w_*(\Omega^\bullet_U)) 
} \]
By Theorem~\ref{AH}, $\alpha$ is an isomorphism. Theorem~\ref{E1D}
implies that $\beta$ is injective. Therefore $\alpha\circ \beta$ injective. 
Therefore $\alpha^n$ is injective.

But $F^n\Omega^\bullet_X(\log \Sigma)=\Omega^n_X(\log\Sigma)[-n]$
and $F^nw_*(\Omega^\bullet_U)=w_*(\Omega^n_U)[-n]$. Therefore $\alpha^n$
becomes
$$
\alpha^n\colon H^q(X,\Omega^n_X(\log\Sigma))\to H^q(X,w_*(\Omega^n_U))
$$
The morphism $w\colon U\subset X$ is affine, so
$H^q(X,w_*(\Omega^n_U) )\to H^q(U,\Omega^n_U)$ is an isomorphism.
Therefore $\alpha^n$ becomes the restriction map
$$
\alpha^n\colon H^q(X,\Omega^n_X(\log\Sigma))\to H^q(U,\Omega^n_U).
$$
\end{proof}

\begin{cor}\label{oit} Let $T$ be a $\Q$-divisor on $X$ 
such that $T\sim_\Q 0$ and $\Supp\{T\}\subseteq \Sigma$. 
In particular, $T|_U$ has integer coefficients. 
Then the restriction homomorphism 
$$
H^q(X,\cO_X(K_X+\Sigma+\lfloor T\rfloor))\to H^q(U,\cO_U(K_U+T|_U))
$$ 
is injective, for all $q$.
\end{cor}

\begin{proof} We use the notations of paragraph 1.H.
Denote $V=\tau^{-1}(U)=Y\setminus \Sigma_Y$. 
By Theorem~\ref{oi}, the restriction 
$$
H^q(Y,\cO_Y(K_Y+\Sigma_Y))\to H^q(V,\cO_V(K_V))
$$
is injective. By the Leray spectral sequence and Lemma~\ref{keyct}, 
the restriction
$$
H^q(X,\tau_*\cO_Y(K_Y+\Sigma_Y))\to H^q(U,\tau_*\cO_V(K_V))
$$
is injective. Equivalently, the direct sum of restrictions
$$
\oplus_{i=0}^{r-1} (H^q(X,\cO_X(K_X+\Sigma+\lfloor iT\rfloor))\to H^q(U,\cO_U(K_U+iT|_U))
$$
is injective. For $i=1$, we obtain the claim.
\end{proof}

\begin{thm}\label{oitlog}
Let $X$ be a proper non-singular variety.
Let $U$ be an open subset of $X$ such that $X\setminus U$
is a normal crossings divisor with irreducible components $(E_i)_i$. 
Let $L$ be a Cartier divisor on $X$ such that 
$L\sim_\R K_X+\sum_i b_iE_i$, with $0<b_i\le 1$ for all $i$.
Then the restriction homomorphism
$$
H^q(X,\cO_X(L))\to H^q(U,\cO_U(L|_U))
$$ 
is injective, for all $q$. 
\end{thm}

\begin{proof} Choose a labeling of the components, say $E_1,\ldots,E_l$.
Since $L-K_X$ has integer coefficients, it follows by Lemma~\ref{1.5} that the set 
$$
V=\{x\in \R^l;L\sim_\R K_X+\sum_{i=1}^l x_iE_i\}
$$ 
is a non-empty affine linear subspace of $\R^l$ defined over $\Q$.
Then $(b_1,\ldots,b_l)\in V\cap (0,1]^l$ can be approximated by 
$(b'_1,\ldots,b'_l)\in V\cap (0,1]^l\cap \Q^l$, such that $b'_i=b_i$
if $b_i\in \Q$. Since $0\sim_\R -L+K_X+\sum_i b'_iE_i$ and the right hand side has
rational coefficients, it follows that $0\sim_\Q -L+K_X+\sum_i b'_iE_i$. 

In conclusion, $L\sim_\Q K_X+\sum_i b'_iE_i$ and $0<b'_i\le 1$ for all $i$.
Set $\Sigma=\sum_i E_i$ and $T=L-K_X-\sum_i b'_iE_i$.
Then $T\sim_\Q 0$, $\{T\}=\sum_i \{-b'_i\}E_i$ and 
$L=K_X+\Sigma +\lfloor T\rfloor$. Corollary~\ref{oit} gives the claim.
\end{proof}

\begin{rem}
Let $U\subseteq U'\subseteq X$ be another open subset. From the commutative diagram
\[ 
\xymatrix{
H^q(X,\cO_X(L)) \ar[rr]\ar[dr] & & H^q(U,\cO_U(L|_U)) \\
          & H^q(U',\cO_{U'}(L|_{U'})) \ar[ur] &
} \]
it follows that $H^q(X,\cO_X(L))\to H^q(U',\cO_{U'}(L|_{U'}))$ is injective for all $q$.
\end{rem}

\begin{rem} Recall that for an $\cO_X$-module $\cF$,
$\Gamma_\Sigma(X,\cF)$ is defined as the kernel of 
$\Gamma(X,\cF)\to \Gamma(U,\cF|_U)$.
The functor $\Gamma_\Sigma(X,\cdot)$ is left exact. Its derived functors, 
denoted $(H^i_\Sigma(X,\cF))_{i\ge 0}$, are called the cohomology of 
$X$ modulo $U$, with coefficients in $\cF$. 
For every $\cF$ we have long exact sequences 
$$
0\to \Gamma_\Sigma(X,\cF)\to \Gamma(X,\cF)\to \Gamma(U,\cF|U)\to 
H_\Sigma^1(X,\cF)\to H^1(X,\cF)\to H^1(U,\cF|U)\to\cdots
$$
Therefore Theorem~\ref{oitlog} says that the homomorphism 
$
H^q_\Sigma(X,\cO_X(L))\to H^q(X,\cO_X(L))
$
is zero for all $q$. Equivalently, $\Gamma_\Sigma(X,\cO_X(L))=0$, 
and for all $q$ we have short exact sequences 
$$
0\to H^q(X,\cO_X(L))\to H^q(U,\cO_U(L|_U))\to H^{q+1}_\Sigma(X,\cO_X(L))\to 0.
$$
\end{rem}

\begin{rem}\label{refor} Theorem~\ref{oitlog} is also equivalent to the following 
statement, which generalizes the original result of Esnault and Viehweg~\cite[Theorem 5.1]{EVlect}:
let $D$ be an effective Cartier divisor supported by $\Sigma$. Then the long exact sequence induced
in cohomology by the short exact sequence 
$
0\to \cO_X(L)\to \cO_X(L+D)\to \cO_D(L+D)\to 0
$
breaks up into short exact sequences
$$
0\to H^q(X,\cO_X(L))\to H^q(X,\cO_X(L+D))\to H^q(D,\cO_D(L+D))\to 0 \ (q\ge 0).
$$
Indeed, let $D$ be as above. We have a commutative diagram
\[ 
\xymatrix{
H^q(X,\cO_X(L)) \ar[r]^\alpha\ar[d]^\beta  & H^q(X,\cO_X(L+D))  
\ar[d] \\ 
H^q(U,\cO_U(L|_U)) \ar[r]^\gamma       & H^q(U,\cO_U((L+D)|_U))
} \]
Since $D$ is disjoint from $U$, $\gamma$ is an isomorphism.
By Theorem~\ref{oitlog}, $\beta$ is injective. Therefore $\gamma\circ \beta$
is injective. It follows that $\alpha$ is injective.
Conversely, suppose $H^q(X,\cO_X(L))\to H^q(X,\cO_X(L+D))$ is injective for
all divisors $D$ supported by $X\setminus U$. Then $H^q(X,\cO_X(L))\to H^q(X,\cO_X(L+m\Sigma))$
is injective for every $m\ge 0$. Lemma~\ref{dl} implies the injectivity of 
$$
H^q(X,\cO_X(L))\to \varinjlim_m H^q(X,\cO_X(L+m\Sigma)).
$$
By Lemma~\ref{ob},  this is isomorphic to the homomorphism 
$H^q(X,\cO_X(L))\to H^q(U,\cO_U(L|_U))$.
\end{rem}

\begin{cor}
Let $D$ be an effective Cartier divisor supported by $\Sigma$. 
Then 
$$
0\to H^q(X,\cO_X(K_X+\Sigma))\to H^q(X,\cO_X(K_X+\Sigma+D))\to 
H^q(D,\cO_D(K_X+\Sigma+D))\to 0
$$
is a short exact sequence, for all $q$.
\end{cor}

\begin{proof} By Remark~\ref{refor} for $L=K_X+\Sigma$. 
\end{proof}

\begin{cor}
The homomorphism $\Gamma(X,\cO_X(K_X+2\Sigma))\to \Gamma(\Sigma,
\cO_\Sigma(K_X+2\Sigma))$ is surjective. 
\end{cor}

If $\Sigma$ is the general member of a base point free linear
system, this is the original result of Tankeev~\cite[Proposition 1]{Tank71}.


\section{Differential forms of intermediate degree}


Let $(X,\Sigma)$ be a log smooth pair such that $X$ is proper
and $U=X\setminus \Sigma$ is contained in an affine open subset of $X$.

\begin{thm}\label{ov}
$
H^q(X,\Omega^p_X(\log \Sigma))=0  \text{ for }p+q>\dim X.
$
In particular, $H^q(X,\cO_X(K_X+\Sigma))=0$ for $q>0$.
\end{thm}

\begin{proof}
Consider the logarithmic de Rham complex $\Omega^\bullet_X(\log \Sigma)$. 
Let $U'$ be an affine open subset of $X$ containing $U$. The inclusions 
$U\subseteq U'\subset X$ induce a commutative diagram
\[ 
\xymatrix{
\bH^r(X,\Omega^\bullet_X(\log \Sigma)) \ar[rr]\ar[dr] & & 
\bH^r(U,\Omega^\bullet_U) \\
& \bH^r(U',\Omega^\bullet_X(\log \Sigma)|_{U'}) \ar[ur] &
} \]
Since $U'$ is affine, $H^q(U',\Omega^p_X(\log \Sigma)|_{U'})=0$
for $q>0$. Therefore 
$\bH^r(U',\Omega^\bullet_X(\log \Sigma)|_{U'})$ is the $r$-th
homology of the differential complex 
$\Gamma(U',\Omega^\bullet_X(\log \Sigma))$. 
Since $\Omega^p_X(\log \Sigma)=0$ for $p>\dim X$, we obtain
$$
\bH^r(U',\Omega^\bullet_X(\log \Sigma)|_{U'})=0 \text{ for }
r>\dim X.
$$
Let $r>\dim X$. It follows that the horizontal map is zero.
But it is an isomorphism by Theorem~\ref{AH}. Therefore 
$$
\bH^r(X,\Omega^\bullet_X(\log \Sigma))=0.
$$
By Theorem~\ref{E1D}, we have a non-canonical isomorphism
$$
\bH^r(X,\Omega^\bullet_X(\log \Sigma))\simeq \oplus_{p+q=r}
H^q(X,\Omega^p_X(\log \Sigma))
$$
Therefore $H^q(X,\Omega^p_X(\log \Sigma))=0$ for all $p+q=r$.
\end{proof}

Let $T$ be a $\Q$-divisor on $X$ such that $T\sim_\Q 0$ and 
$\Supp\{T\}\subseteq \Sigma$. In particular, $T|_U$ has 
integer coefficients.

\begin{thm}\label{ovt}
$
H^q(X,\Omega^p_X(\log \Sigma)\otimes\cO_X(\lfloor T\rfloor))=0  
\text{ for }p+q>\dim X.
$
In particular, $H^q(X,\cO_X(K_X+\Sigma+\lfloor T\rfloor))=0$ for $q>0$.
\end{thm}

\begin{proof} We use the notations of paragraph 1.H.
Let $X\setminus \Sigma\subseteq U'$, with $U'$ an 
affine open subset of $X$. Let $V'=\tau^{-1}(U')$.
By Lemma~\ref{keyct}, the Leray spectral sequence associated to
$\tau|_{V'}\colon V'\to U'$ and $\Omega^p_Y(\log \Sigma_Y)|_{V'}$
degenerates into isomorphisms
$$
H^q(U',(\tau|_{V'})_*\Omega^p_Y(\log \Sigma_Y)|_{V'})\isoto 
H^q(V',\Omega^p_Y(\log \Sigma_Y)|_{V'}).
$$
Since $U'$ is affine, the left hand side is zero for $q>0$.
Therefore 
$$
H^q(V',\Omega^p_Y(\log \Sigma_Y)|_{V'})=0 \text{ for }q>0.
$$
In particular, the spectral sequence 
$$
E_1^{pq}=H^q(V',\Omega^p_Y(\log \Sigma_Y)|_{V'})\Longrightarrow
\bH^q(V',\Omega^\bullet_Y(\log \Sigma_Y)|_{V'})
$$
degenerates into isomorphisms
$$
h^r(\Gamma(V',\Omega^\bullet_Y(\log \Sigma_Y)))\simeq 
\bH^r(V',\Omega^\bullet_Y(\log \Sigma_Y)|_{V'}),
$$
where the first term is the $r$-th homology group of the 
differential complex $\Gamma(V',\Omega^\bullet_Y(\log \Sigma_Y))$. 
Since $\Omega^p_Y(\log \Sigma_Y)=0$ for $p>\dim Y$, we obtain 
$$
\bH^r(V',\Omega^\bullet_Y(\log \Sigma_Y)|_{V'})=0\text{ for }r>\dim Y.
$$
Let $V=\tau^{-1}(U)=Y\setminus \Sigma_Y$. The restriction map 
$$
\bH^r(Y,\Omega^\bullet_Y(\log \Sigma_Y))\to
\bH^r(V,\Omega^\bullet_Y(\log \Sigma_Y)|_V)
$$
is an isomorphism by Theorem~\ref{AH}. It factors through 
$\bH^r(V',\Omega^\bullet_Y(\log \Sigma_Y)|_{V'})$, hence it is zero
for $r>\dim Y$. Therefore 
$$
\bH^r(Y,\Omega^\bullet_Y(\log \Sigma_Y))=0 \text{ for }r>\dim Y.
$$
By Theorem~\ref{E1D}, 
$
\bH^r(Y,\Omega^\bullet_Y(\log \Sigma_Y))\simeq \oplus_{p+q=r}H^q(Y,
\Omega^p_Y(\log \Sigma_Y)).
$
Therefore 
$$
H^q(Y,\Omega^p_Y(\log \Sigma_Y))=0\text{ for }p+q>\dim Y.
$$
The cyclic group of order $r$ acts on $H^q(Y,\Omega^p_Y(\log \Sigma_Y))$,
with eigenspace decomposition
$$
\oplus_{i=0}^{r-1}H^q(X,\Omega^p_X(\log \Sigma)\otimes\cO_X(\lfloor iT\rfloor)).
$$
Therefore $H^q(X,\Omega^p_X(\log \Sigma)\otimes\cO_X(\lfloor T\rfloor))=0$.
\end{proof}


\subsection{Applications}


\begin{cor}\label{ovdual}
$
H^q(X,\Omega^p_X(\log \Sigma)\otimes\cO_X(-\Sigma-\lfloor T\rfloor))=0  
\text{ for }p+q<\dim X.
$
In particular, $H^q(X,\cO_X(-\Sigma-\lfloor T\rfloor))=0$ for all $q<\dim X$.
\end{cor}

\begin{proof} This is the dual form of Theorem~\ref{ovt}, using Serre duality
and the isomorphism 
$(\Omega_X^p(\log \Sigma))^\vee\simeq \Omega_X^{\dim X-p}(\log\Sigma)\otimes 
\cO_X(-K_X-\Sigma)$.
\end{proof}

For $T=0$, we obtain
$
H^q(X,\cI_\Sigma\otimes \Omega^p_X(\log \Sigma))=0  \text{ for all }p+q<\dim X.
$
In particular, $H^q(X,\cI_\Sigma)=0$ for all $q<\dim X$.

\begin{cor} The homomorphism
$
H^q(X,\Omega^p_X\otimes\cO_X(-\lfloor T\rfloor))\to 
H^q(\Sigma,\tilde{\Omega}^p_\Sigma\otimes \cO_\Sigma(-\lfloor T\rfloor))
$
is bijective for $p+q<\dim \Sigma$ and injective for $p+q=\dim \Sigma$.
\end{cor}

\begin{proof} Denote 
$K^{pq}=H^q(X,\Omega_X^p(\log \Sigma)\otimes\cO_X(-\Sigma-\lfloor T\rfloor))$.
The short exact sequence of Lemma~\ref{rs} induces 
a long exact sequence in cohomology
$$
K^{pq}\to
H^q(X,\Omega_X^p\otimes\cO_X(-\lfloor T\rfloor))\stackrel{\alpha^{qp}}{\to} 
H^q(\Sigma,\tilde{\Omega}^p_\Sigma\otimes \cO_\Sigma(-\lfloor T\rfloor))\to
K^{p,q+1}
$$
By Corollary~\ref{ovdual}, $\alpha^{qp}$ is bijective for 
$q+1<\dim X-p$, and injective for $q+1=\dim X-p$. 
\end{proof}

\begin{cor}[Weak Lefschetz] The restriction homomorphism 
$
H^r_{DR}(X/k)\to H^r_{DR}(\Sigma/k)
$
is bijective for $r<\dim \Sigma$ and injective for $r=\dim \Sigma$.
\end{cor}

\begin{proof} Set $T=0$. The homomorphism
$
H^q(X,\Omega^p_X) \to H^q(\Sigma,\tilde{\Omega}^p_\Sigma)
$
is bijective for $p+q<\dim \Sigma$ and injective for $p+q=\dim \Sigma$.
The Hodge to de Rham spectral sequence degenerates at $E_1$, for
$X/k$ by~\cite[Theorem 5.5]{Del69A} and for $\Sigma/k$ by Theorem~\ref{NCHR}, 
and is compatible with the maps above.
\end{proof}

\begin{cor}
Suppose $\Supp\{T\}=\Sigma$. Then 
$H^q(X,\Omega^p_X(\log \Sigma)\otimes\cO_X(\lfloor T\rfloor))=0$
for all $p+q\ne \dim X$.
\end{cor}

\begin{proof}
For $p+q>\dim X$, this follows from above. For $p+q<\dim X$, apply
the dual form to $-T$, using $-\Sigma-\lfloor -T\rfloor=\lfloor T\rfloor$.
\end{proof}

\begin{cor}
Suppose $X\setminus \Supp\{T\}$ is contained in an affine open subset
of $X$. Then $H^q(X,\cO_X(\lfloor T\rfloor))=0$ for $q< \dim X$.
\end{cor}

\begin{thm}[Akizuki-Nakano]
Let $X$ be projective non-singular variety. 
Let $L$ be an ample divisor. Then $H^q(X,\Omega^p_X(L))=0$ for $p+q>\dim X$.
Dually, $H^q(X,\Omega^p_X(-L))=0$ for $p+q<\dim X$.
\end{thm}

\begin{proof}
There exists $r\ge 1$ such that the general member 
$Y\in |rL|$ is non-singular. Set $T=L-\frac{1}{r}Y$ and 
$\Sigma=Y$. Then $T\sim_\Q 0$, $\Supp\{T\}=\Sigma$ and 
$X\setminus \Sigma$ is affine. We also have 
$\lfloor T\rfloor=L-Y$. By Theorem~\ref{ovt}, we obtain
$$
H^q(X,\Omega^p_X(\log Y)\otimes \cO_X(L-Y))=0\text{ for }
p+q>\dim Y.
$$
The short exact sequence of Lemma~\ref{rs}, tensored by $L$,
gives an exact sequence 
$$
H^q(X,\Omega^p_X(\log Y)(L-Y))\to H^q(X,\Omega^p_X(L))\to H^q(Y,\Omega^p_Y(L)).
$$
Let $p+q>\dim X$. The first term is zero from above, and the
third is zero by induction. Therefore 
$H^q(X,\Omega^p_X\otimes \cO_X(L))=0$.
\end{proof}

\begin{cor}[Kodaira]
Let $X$ be projective non-singular variety. 
Let $L$ be an ample divisor on $X$. Then $H^q(X,\cO_X(K_X+L))=0$ for $q>0$.
\end{cor}


\section{Log pairs}


A {\em log pair} $(X,B)$ consists of a normal algebraic variety $X$, 
endowed with an $\R$-Weil divisor $B$ such that $K_X+B$ is $\R$-Cartier.
If $B$ is effective, we call $(X,B)$ a {\em log variety}.

A contraction $f\colon X\to Y$ is a proper morphism such that the natural
homomorphism $\cO_Y\to f_*\cO_X$ is an isomorphism.


\subsection{Totally canonical locus}

Let $(X,B)$ be a log pair. Let $\mu\colon X'\to X$ be a birational contraction
such that $(X',\Exc(\mu)\cup \Supp \mu^{-1}_*B)$ is log smooth. 
Let 
$$
\mu^*(K_X+B)=K_{X'}+B_{X'}
$$ 
be the induced log pair structure on $X'$. We say that $\mu\colon (X',B_{X'})\to (X,B)$
is a log crepant birational contraction.

For a prime divisor $E$ on $X'$, $1-\mult_E(B_{X'})$ is called the log discrepancy
of $(X,B)$ in the valuation of $k(X)$ defined by $E$, denoted $a(E;X,B)$
(see~\cite{Amb06} for example).

Define an open subset of $X$ by the formula
$
U=X\setminus \mu(\Supp (B_{X'})^{>0}).
$
The definition of $U$ does not depend on the choice of $\mu$, by the following

\begin{lem}
Let $\mu\colon (X',B')\to (X,B)$ be a log crepant proper birational morphism of
log pairs with log smooth support. Then $\mu(\Supp {B'}^{>0})=\Supp B^{>0}$.
\end{lem}

\begin{proof} First, we claim that $B'\le \mu^*B$. Indeed,
$X$ is non-singular, so $K_{X'}-\mu^*K_X$ is effective $\mu$-exceptional.
From $\mu^*(K_X+B)=K_{X'}+B'$ we obtain
$$
\mu^*B-B'=K_{X'}-\mu^*K_X\ge 0.
$$
To prove the statement, denote $U=X\setminus \Supp(B^{>0})$.
Then $B|_U\le 0$. The claim for 
$\mu|_{\mu^{-1}(U)}\colon (\mu^{-1}(U),B'|_{\mu^{-1}(U)})\to (U,B|_U)$
gives $B'|_{\mu^{-1}(U)}\le 0$. Therefore 
$\mu(\Supp B'^{>0})\subseteq \Supp B^{>0}$. For the opposite
inclusion, note that $\Supp B^{>0}$ has codimension one. Let 
$E$ be a prime in $\Supp B^{>0}$. Since $\mu$ is an isomorphism 
in a neighbourhood of the generic point of $E$, $E$ also appears
as a prime on $X'$ and $\mult_E(B')=\mult_E(B)>0$. Therefore
$E\subseteq \mu(\Supp B'^{>0})$. 
\end{proof}

We call $U$ the {\em totally canonical locus} of $(X,B)$.
It is the largest open subset $U$ of $X$ with the property that every
geometric valuation over $U$ has log discrepancy at least $1$ with
respect to $(U,B|_U)$. We have 
$$
X\setminus (\Sing(X)\cup \Supp(B^{>0}))\subseteq U\subseteq 
X\setminus \Supp(B^{>0}).
$$
The first inclusion implies that $U$ is dense in $X$.
The second inclusion is an equality if $(X,\Supp B)$ is log
smooth.


\subsection{Non-log canonical locus}

Let $(X,B)$ be a log pair with log smooth support. Write 
$B=\sum_E b_E E$, where the sum runs after the prime divisors 
of $X$. Define
$$
N(B)=\sum_{b_E<0}\lfloor b_E\rfloor E+\sum_{b_E>1}(\lceil b_E\rceil-1)E.
$$
Then $N(B)$ is a Weil divisor. There exists a unique decomposition 
$N(B)=N^+-N^-$, where $N^+,N^-$ are effective divisors with no 
components in common. Then $\Supp(N^+)=\Supp(B^{>1})$ and 
$\Supp(N^-)=\Supp(B^{<0})$. We have 
$$
\lfloor B^{>1}\rfloor-N^+=\sum_{0<b_E\in \Z}E.
$$
In particular $N^+\le \lfloor B^{>1}\rfloor$, and
the two divisors have the same support. Denote 
$$
\Delta(B)=B-N(B).
$$
We have $\Delta(B)=\sum_{b_E<0}\{ b_E\} E+\sum_{b_E>0}(b_E+1-\lceil b_E\rceil)E$.
The following properties hold:
\begin{itemize}
\item[1)] The coefficients of $\Delta(B)$ belong to the interval $[0,1]$.
They are rational if and only if the coefficients of $B$ are.
\item[2)] $\Supp(\Delta(B))=\Supp(B^{>0})\cup \cup_{0>b_E\notin \Z}E$.
In particular, $(X,\Delta(B))$ is a log variety with log canonical singularities
and log smooth support.
\item[3)] $\mult_E\Delta(B)=1$ if and only if $\mult_EB\in \Z_{>0}$.
\end{itemize}

\begin{lem}\label{ninv}
Let $\mu\colon (X',B')\to (X,B)$ be a log crepant birational contraction of
log pairs with log smooth support. Then $\mu^*N(B)-N(B')$ is an effective 
$\mu$-exceptional divisor. In particular,
$$
\cO_X(-N(B))=\mu_*\cO_{X'}(-N(B')).
$$
\end{lem}

\begin{proof} The operation $B\mapsto N(B)$ is defined
componentwise, so $\mu^*N(B)-N(B')$ is clearly $\mu$-exceptional.
Decompose $B=\Delta+N$ and $B'=\Delta'+N'$. From $\mu^*(K+B)=K_{X'}+B'$
we deduce 
$$
\mu^*N-N'=K_{X'}+\Delta'-\mu^*(K+\Delta).
$$
In particular, let $E$ be a prime divisor on $X'$. And 
$m_E=\mult_E(\mu^*N-N')$. Then 
$$
m_E=a(E;X,\Delta)-a(E;X',\Delta').
$$
Since $(X,\Delta)$ has log canonical singularities and $\Delta'$ is effective, we obtain 
$$
m_E\ge 0-1\ge -1.
$$
If $m_E>-1$, then $m_E\ge 0$, as it is an integer. Else, $m_E=-1$.
In this case $a(E;X,\Delta)=0$ and $a(E;X',\Delta')=1$. 
From $a(E;X,\Delta)=0$, we deduce that $\mu(E)$ is the transverse
intersection of some components of $\Delta$ with coefficient $1$.
That is $\mu(E)$ is the transverse intersection of some components
of $B$ with coefficients in $\Z_{\ge 1}$. In particular, $B\ge \Delta$
near the generic point of $\mu(E)$. We deduce
$$
0=a(E;X,\Delta)\ge a(E;X,B)=a(E;X',B')
$$ 
That is $\mult_E B'\ge 1$. Then $\mult_E\Delta'>0$, so 
$a(E;X',\Delta')=1-\mult_E\Delta'<1$. Contradiction.
\end{proof}

\begin{defn}
Let $(X,B)$ be a log variety. Let $\mu\colon (X',B_{X'})\to (X,B)$
be a log crepant log resolution. Define 
$$
\cI=\mu_*\cO_{X'}(-N(B_{X'})).
$$
The coherent $\cO_X$-module $\cI$ is independent of the choice of $\mu$,
by Lemma~\ref{ninv}. Since $B$ is effective, the divisor
$N(B_{X'})^-=-\lfloor B_{X'}^{<0}\rfloor$ is $\mu$-exceptional.
Therefore 
$$
\cI\subseteq \mu_*\cO_{X'}(N(B_{X'})^-)=\cO_X.
$$
We call $\cI$ the {\em ideal sheaf of the non-log canonical locus} of 
$(X,B)$. It defines a closed subscheme $(X,B)_{-\infty}$ of $X$ by the
short exact sequence 
$$
0\to \cI\to \cO_X\to \cO_{(X,B)_{-\infty}}\to 0.
$$
We call $(X,B)_{-\infty}$ the 
{\em locus of non-log canonical singularities} of $(X,B)$.
It is empty if and only if $(X,B)$ has log canonical singularities.
The complement $X\setminus (X,B)_{-\infty}$ is the largest open 
subset on which $(X,B)$ has log canonical singularities.
\end{defn}

\begin{rem}\label{comp}
We introduced in~\cite{Amb03} another scheme structure on the locus
of non-log canonical singularities of a log variety $(X,B)$. The
two schemes have the same support, but their structure sheaves usually 
differ. To compare them, consider a log crepant log resolution 
$\mu\colon (X',B_{X'})\to (X,B)$. Define 
$$
N^s=\lfloor B_{X'}^{\ne 1}\rfloor=N(B_{X'})+\sum_{\mult_E(B_{X'})\in \Z_{>1}}E.
$$ 
Denote $B_{X'}=\sum_E b_E E$. Then 
$N^s-N(B_{X'})=\sum_{b_E\in \Z_{>1}}E$ and $\lfloor B_{X'}\rfloor-N^s=\sum_{b_E=1}E$.
In particular
$$
N\le N^s\le \lfloor B_{X'}\rfloor.
$$
We obtain inclusions of ideal sheaves
$\mu_*\cO_{X'}(-N)\supseteq \mu_*\cO_{X'}(-N^s)\supseteq
\mu_*\cO_{X'}(-\lfloor B_{X'}\rfloor)$. Equivalently, we have closed
embeddings of subschemes of $X$
$$
Y\hookrightarrow Y^s\hookrightarrow \LCS(X,B),
$$
where $Y^s$ is the scheme structure introduced in~\cite{Amb03} and
$\LCS(X,B)$ is the subscheme structure on the non-klt locus of
$(X,B)$.

Consider for example the log variety $(\bA^2,2H_1+H_2)$, where
$H_1,H_2$ are the coordinate hyperplanes. The above inclusions are
$$
H_1\hookrightarrow 2H_1\hookrightarrow 2H_1+H_2. 
$$
\end{rem}

\begin{lem}
Let $\mu\colon (X',B')\to (X,B)$ be a log crepant birational contraction of
log pairs with log smooth support. 
Then $\mu^*\lfloor B^{\ne 1}\rfloor-\lfloor B_{X'}^{\ne 1}\rfloor$ is an effective 
$\mu$-exceptional divisor. In particular,
$$
\cO_X(-\lfloor B^{\ne 1}\rfloor)=\mu_*\cO_{X'}(-\lfloor B_{X'}^{\ne 1}\rfloor).
$$
\end{lem}

\begin{proof} The operation $B\mapsto \lfloor B^{\ne 1}\rfloor$ is defined
componentwise, so $\mu^*\lfloor B^{\ne 1}\rfloor-\lfloor B_{X'}^{\ne 1}\rfloor$ 
is clearly $\mu$-exceptional. The equality $\mu^*(K+B)=K_{X'}+B_{X'}$ becomes
$$
\mu^*\lfloor B^{\ne 1}\rfloor-\lfloor B_{X'}^{\ne 1}\rfloor=
K_{X'}+B_{X'}^{=1}+\{B_{X'}^{\ne 1}\}-\mu^*(K+B^{=1}+\{B^{\ne 1}\}).
$$
Consider the multiplicity of the left hand side at a prime on $X'$.
It is an integer. The right hand side is $\ge -1$. If $>-1$, it is
$\ge 0$. Suppose it equals $-1$. This implies
$a(E;X,B^{=1}+\{B^{\ne 1}\})=0$. Then $a(E;X,B^{=1})=0$ and $B=B^{=1}$
near the generic point of $\mu(E)$. Then $a(E;X',B_{X'})=0$.
Then the difference is zero. Contradiction.
\end{proof}


\subsection{Lc centers}

For the definition and properties of lc centers, see~\cite{Amb06}. 

\begin{lem}\label{ad}
Let $(X,B)$ be a log variety with log canonical singularities.
Let $D$ be an effective $\R$-Cartier $\R$-divisor on $X$, let
$Z$ be the union of lc centers of $(X,B)$ contained in $\Supp D$, 
with reduced structure. Then $(X,B+\epsilon D)_{-\infty}=Z$ for $0<\epsilon\ll 1$.
\end{lem}

\begin{proof}
Let $\mu\colon X'\to X$ be a resolution of singularities such that
$(X',\Supp B_{X'}\cup \Supp \mu^*D)$ is log smooth, where 
$\mu^*(K_X+B)=K_{X'}+B_{X'}$, and $\mu^{-1}(Z)$ has pure codimension one. 
We have $\mu^*(K_X+B+\epsilon D)=K_{X'}+B_{X'}+\epsilon\mu^*D$. Denote
$$
\Sigma'=\sum_{\mult_E(B_{X'})=1, \mu(E)\subseteq Z} E.
$$
Since the coefficients of $B_{X'}$ are at most $1$, for $0<\epsilon\ll 1$ we obtain the formula
\begin{align*}
N(B_{X'}+\epsilon\mu^*D) & =\lfloor (B_{X'})^{<0}\rfloor+ \sum_{\mult_E(B_{X'})=1, \mu(E)\subseteq \Supp D} E \\
 & =\lfloor (B_{X'})^{<0}\rfloor+ \Sigma'.
\end{align*}

Denote $A=-\lfloor (B_{X'})^{<0}\rfloor$, an effective $\mu$-exceptional divisor on $X'$.
Consider the commutative diagram with exact rows
\[ 
\xymatrix{
0 \ar[r] & \mu_*\cO_{X'}(A-\Sigma')\ar[r] & \mu_*\cO_{X'}(A) \ar[r]^r & \mu_*\cO_{\Sigma'}(A|_{\Sigma'})\ar[r]^\partial & R^1\mu_*\cO_{X'}(A-\Sigma') \\ 
0 \ar[r] & \cI_Z\ar[r] \ar[u]^\alpha & \cO_X\ar[r] \ar[u]^\beta & \cO_Z\ar[r] \ar[u]^\gamma & 0
} \]

We claim that $\partial=0$. Indeed, denote $B'=\{B_{X'}^{<0}\}+B_{X'}^{>0}-\Sigma'$.
Then $A-\Sigma'\sim_\R K_{X'}+B'$ over $X$, $(X',B')$ has log canonical singularities, and 
$\mu(C)\nsubseteq Z$ for every lc center $C$ of $(X',B')$. The sheaf $\mu_*\cO_{\Sigma'}(A|_{\Sigma'})$
is supported by $Z$, so the image of $\partial$ is supported by $Z$. Suppose by contradiction
that $\partial$ is non-zero. Let $s$ be a non-zero local section of $\im \partial$.
By~\cite[Theorem 3.2.(i)]{Amb03}, $(X',B')$ admits an lc center $C$ such that $\mu(C)\subseteq \Supp(s)$.
Since $\Supp(s)\subseteq Z$, we obtain $\mu(C)\subseteq Z$, a contradiction.

Since $A$ is effective and $\mu$-exceptional, $\beta$ is an isomorphism. The map $\gamma$
is injective. Since $r$ is surjective, $\gamma$ is also surjective, hence an isomorphism. We conclude that $\alpha$
is an isomorphism. That is $\cI_Z=\mu_*\cO_{X'}(-N(B_{X'}+\epsilon\mu^*D))=\cI_{(X,B+\epsilon D)_{-\infty}}$.
\end{proof}


\section{Injectivity for log varieties}


\begin{thm}\label{loginj}
Let $(X,B)$ be a proper log variety with log canonical singularities. 
Let $U$ be the totally canonical locus of $(X,B)$. 
Let $L$ be a Cartier divisor on $X$ such that $L\sim_\R K+B$.
Then the restriction homomorphism
$$
H^1(X,\cO_X(L))\to H^1(U,\cO_U(L|_U))
$$
is injective.
\end{thm}

\begin{proof}
Let $\mu\colon X'\to X$ be a birational contraction such that
$X'$ is non-singular, the exceptional locus $\Exc\mu$ has codimension
one, and $\Exc\mu\cup \Supp(\mu^{-1}_*B)$ has normal crossings.
We can write 
$$
K_{X'}+\mu^{-1}_*B+\Exc\mu=\mu^*(K+B)+A,
$$
with $A$ supported by $\Exc\mu$. Since $(X,B)$ has log canonical 
singularities, $A$ is effective. Denote $B'=\mu^{-1}_*B+\Exc\mu-\{A\}$
and $L'=\mu^*L+\lfloor A\rfloor$. We obtain
$$
L'\sim_\R K_{X'}+B'.
$$
Denote $U'=X'\setminus B'$. We claim that $U'\subseteq \mu^{-1}(U)$.
Indeed, this is equivalent to the inclusion 
$$
\Supp(B')\supseteq \mu^{-1}\mu(\Supp B_{X'}^{>0}).
$$
By Zariski's Main Theorem, $\Exc\mu=\mu^{-1}(X\setminus V)$, where 
$V$ is the largest open subset of $X$ such that $\mu$ is an isomorphism
over $V$. Over $X\setminus V$, the inclusion is clear since 
$\Exc\mu\subseteq \Supp B'$. Over $V$, $\mu$ is an isomorphism 
and the inclusion becomes an equality. This proves the claim.

Since $A$ is effective and $\mu_*A=0$, we have 
$\cO_X(L)\isoto\mu_*\cO_{X'}(L')$.
From $U'\subseteq \mu^{-1}(U)$ we obtain a commutative diagram
\[ 
\xymatrix{
H^1(X',\cO_{X'}(L')) \ar[rr]^{\alpha'} & & H^1(U',\cO_{U'}(L'|_{U'}))  \\ 
H^1(X,\cO_X(L)) \ar[rr]^\alpha \ar[u]^\beta & & H^1(U,\cO_U(L|_U))\ar[u]
} \]
By Theorem~\ref{oitlog}, $\alpha'$ is injective. Since 
$\cO_X(L)=\mu_*\cO_{X'}(L')$, Lemma~\ref{i1} implies that
$\beta$ is injective. Then
$\alpha'\circ \beta$ is injective. The diagram is commutative,
so $\alpha$ is injective.
\end{proof}

\begin{cor}\label{wd} In the assumptions of Theorem~\ref{loginj}, let $D$
be an effective Cartier divisor such that $\Supp(D)\cap U=\emptyset$.
Then we have a short exact sequence 
$$
0\to \Gamma(X,\cO_X(L))\to \Gamma(X,\cO_X(L+D))\to \Gamma(D,\cO_D(L+D))\to 0.
$$
\end{cor}

\begin{proof} Consider the commutative diagram
\[ 
\xymatrix{
H^1(X,\cO_X(L)) \ar[r]^\alpha\ar[d]^\beta  & H^1(X,\cO_X(L+D))  
\ar[d] \\ 
H^1(U,\cO_U(L|_U)) \ar[r]^\gamma       & H^1(U,\cO_U((L+D)|_U))
} \]
Since $D$ is disjoint from $U$, $\gamma$ is an isomorphism.
Since $\beta$ is injective, we obtain that $\gamma\circ \beta$
is injective. Therefore $\alpha$ is injective. The long exact sequence induced
in cohomology by the short exact sequence 
$
0\to \cO_X(L)\to \cO_X(L+D)\to\cO_D(L+D)\to 0
$
gives the claim.
\end{proof}


\subsection{Applications}


Let $(X,B)$ be a proper log variety with log canonical singularities, let $L,H$ be Cartier divisors 
on $X$.

\begin{cor}\label{afv} Suppose $L\sim_\R K_X+B$. 
Suppose the totally canonical locus of $(X,B)$ is contained in some 
affine open subset $U'\subseteq X$. Then $H^1(X,\cO_X(L))=0$.
\end{cor}

\begin{proof} Let $U$ be the totally canonical locus of $(X,B)$. The restriction homomorphism 
$H^1(X,\cO_X(L))\to H^1(U,\cO_U(L|_U))$ is injective. 
It factors through $H^1(U',\cO_{U'}(L|_{U'}))=0$, hence it is zero. Therefore $H^1(X,\cO_X(L))=0$.
\end{proof}

\begin{cor}\label{tfcpbase}
Let $L\sim_\R K_X+B$. Let $H$ be a Cartier divisor on $X$ such that the linear 
system $|nH|$ is base point free for some positive integer $n$.
Let $m_0\ge 1$ and $s\in \Gamma(X,\cO_X(m_0H))$ such that 
$s|_C\ne 0$ for every lc center of $(X,B)$. Then the multiplication 
$$
\otimes s\colon H^1(X,\cO_X(L+mH))\to H^1(X,\cO_X(L+(m+m_0)H))
$$
is injective for $m\ge 1$.
\end{cor}

\begin{proof} Let $D$ be the zero locus of $s$. There exists a rational number
$0<\epsilon<\frac{1}{m_0}$ such that $(X,B+\epsilon D)$ has log canonical
singularities. We have 
$$
L+mH\sim_\R K_X+B+\epsilon D+(m-\epsilon m_0)H.
$$
There exists $n\ge 1$ such that the linear system 
$|n(m-\epsilon m_0)H|$ has no base points. Let $Y$ be a general member,
and denote $B'=B+\epsilon D+\frac{1}{n}Y$. Then $(X,B')$ has log canonical
singularities, $\Supp D\subseteq \Supp B'$ and 
$$
L+mH\sim_\R K_X+B'
$$
Since $\Supp(D)$ is disjoint from the totally canonical locus of $(X,B')$, 
Corollary~\ref{wd} gives the injectivity of $H^1(X,L+mH)\to H^1(X,L+mH+D)$.
\end{proof}

\begin{cor} Let $V\subseteq \Gamma(X,\cO_X(H))$ be a vector subspace such 
that $V\otimes_k \cO_X\to \cO_X(H)$ is surjective.
If $L\sim_\R K+B+tH$ and $t>\dim_k V$, then the multiplication map
$$
V\otimes_k \Gamma(X,\cO_X(L-H))\to \Gamma(X,\cO_X(L))
$$
is surjective.
\end{cor}

\begin{proof} We use induction on $\dim V$. If $\dim V=1$,
then $V=k \varphi$, with $\varphi\colon \cO_X\isoto \cO_X(H)$.
Then $\otimes \varphi\colon \cO_X(L-H)\to\cO_X(L)$ is an isomorphism,
so the claim holds.

Let $\dim V>1$. Let $\varphi\in V$ be a general element, let $Y=(\varphi)+H$. 
Then the claim is equivalent to the surjectivity of the homomorphism 
$$
V|_Y\otimes \Gamma(X,\cO_X(L-H))|_Y\to \Gamma(X,\cO_X(L))|_Y
$$
where $\Gamma(X,\cF)|_Y$ denotes the image of the restriction map 
$\Gamma(X,\cF)\to \Gamma(Y,\cF\otimes \cO_Y)$, and $V|_Y$ is the image
of $V$ under this restriction for $\cF=\cO_X(H)$.

Assuming $\Gamma(X,\cO_X(L-H))|_Y=\Gamma(Y,\cO_Y(L))$ and 
$\Gamma(X,\cO_X(L-H))|_Y=\Gamma(Y,\cO_Y(L-H))$,
we prove the claim as follows: we have $L\sim_\R K_X+B+Y+(t-1)H$.
By adjunction, using that $Y$ is general, we have 
$L|_Y\sim_\R K_Y+B|_Y+(t-1)H|_Y$, $(Y,B|_Y)$ has log canonical
singularities, and $t-1>\dim V-1=\dim V|_Y$. Therefore 
$V|_Y\otimes \Gamma(Y,\cO_Y(L-H))\to \Gamma(Y,\cO_Y(L))$ is surjective by induction.

It remains to show that $\Gamma(X,\cO_X(L))\to \Gamma(Y,\cO_Y(L))$ and 
$\Gamma(X,\cO_X(L-H))\to \Gamma(Y,\cO_Y(L-H))$ are surjective. Consider the second
homomorphism. We have 
$$
L-Y\sim_\R K_X+B+(t-1)H=K_X+B+\epsilon Y+(t-1-\epsilon)H.
$$
Since $Y$ is general, $(X,B+\epsilon Y)$ has log canonical singularities
for $0<\epsilon \ll 1$. Since $H$ is free, we deduce that 
$L-Y\sim_\R K_X+B'$ with $(X,B')$ having log canonical singularities,
and $Y\subseteq \Supp B'$. By Corollary~\ref{wd}, 
$\Gamma(X,\cO_X(L))\to \Gamma(Y,\cO_Y(L))$ is surjective.
The surjectivity of the other homomorphism is proved in the same way.
\end{proof}


\section{Restriction to the non-log canonical locus}


Let $(X,B)$ be a proper log variety, and
$L$ a Cartier divisor on $X$ such that $L\sim_\R K_X+B$.
Suppose the locus of non-log canonical 
singularities $Y=(X,B)_{-\infty}$ is non-empty.

\begin{lem} Suppose $(X,\Supp B)$ is log smooth. 
\begin{itemize}
\item[1)] The long exact sequence induced in cohomology by the 
short exact sequence 
$$
0\to \cI_Y(L) \to \cO_X(L)\to \cO_Y(L)\to 0
$$
breaks up into short exact sequences
$$
0\to H^q(X,\cI_Y(L))\to H^q(X,\cO_X(L))\to H^q(Y,\cO_Y(L))\to 0\ (q\ge 0).
$$ 
\item[2)] Let $E$ be a prime divisor on $X$ such that
$\mult_EB=1$. The long exact sequence induced in cohomology by the 
short exact sequence 
$$
0\to \cI_Y(L-E) \to \cO_X(L-E)\to \cO_Y(L-E)\to 0
$$
breaks up into short exact sequences
$$
0\to H^q(X,\cI_Y(L-E))\to H^q(X,\cO_X(L-E))\to H^q(Y,\cO_Y(L-E))\to 0
\ (q\ge 0).
$$ 
\end{itemize}
\end{lem}

\begin{proof} 1) Let $N=N(B)$, so that $\cI_Y=\cO_X(-N)$. 
We have $L-N\sim_\R K_X+\Delta$ and $N$ is supported by
$\Delta$. By Remark~\ref{refor}, the natural map 
$H^q(X,\cO_X(L-N))\to H^q(X,\cO_X(L))$ is injective for all $q$.

2) We have $L-E\sim_\R K_X+B-E$ and $(X,B-E)_{-\infty}=(X,B)_{-\infty}=Y$.
Therefore 2) follows from 1).
\end{proof}

\begin{thm}[Extension from non-lc locus]\label{8.2} We have a short exact sequence
$$
0\to \Gamma(X,\cI_Y(L))\to \Gamma(X,\cO_X(L))\to \Gamma(Y,\cO_Y(L))\to 0.
$$
\end{thm}

\begin{proof} Let $\mu\colon (X',B_{X'})\to (X,B)$ be a log
crepant log resolution. 
Let $N(B_{X'})=N=N^+-N^-$ and $\Delta=B_{X'}-N(B_{X'})$. We have 
$$
\mu^*L-N\sim_\R K_{X'}+\Delta
$$
and $N^+$ is supported by $\Delta$. By Remark~\ref{refor},
we obtain for all $q$ short exact sequences 
$$
0\to H^q(X',\cO_{X'}(\mu^*L-N))\to H^q(X',\cO_{X'}(\mu^*L+N^-))\to 
H^q(N',\cO_{N^+}(\mu^*L+N^-))\to 0
$$
By definition, $\cI_Y=\mu_*\cO_{X'}(-N)$. Thus
$\cI_Y(L)=\mu_*\cO_{X'}(\mu^*L-N)$, and we obtain a commutative
diagram
\[ 
\xymatrix{
H^q(X',\cO_{X'}(\mu^*L+N^--N^+)) \ar[rr]^{\gamma^q} & & 
H^q(X',\cO_{X'}(\mu^*L+N^-))   \\ 
H^q(X,\cI_Y(L)) \ar[rr]^{\alpha^q} \ar[u]^{\beta^q} & &
H^q(X,\cO_X(L))\ar[u]
} \]
From above, $\gamma^q$ is injective. By Lemma~\ref{i1}, $\beta^1$ is injective.
Therefore $\gamma^1\circ \beta^1$ is injective. Therefore $\alpha^1$
is injective, which is equivalent to our statement.
\end{proof}


\subsection{Applications}

The first application was first stated by Shokurov,
who showed that it follows from the Log Minimal Model Program and Log Abundance 
in the same dimension (see the proof of~\cite[Lemma 10.15]{Sh03}).

\begin{thm}[Global inversion of adjunction]\label{gia} 
Let $(X,B)$ be a proper connected log variety such that $K_X+B\sim_\R 0$. 
Suppose $Y=(X,B)_{-\infty}$ is non-empty. Then $Y$ is connected, and intersects 
every lc center of $(X,B)$.
\end{thm}

\begin{proof}
By Theorem~\ref{8.2}, we have a short exact sequence
$$
0\to \Gamma(X,\cI_Y)\to \Gamma(X,\cO_X)\to \Gamma(Y,\cO_Y)\to 0.
$$
We have $0=\Gamma(X,\cI_Y), k\isoto \Gamma(X,\cO_X)$. Therefore
$k\isoto \Gamma(Y,\cO_Y)$, so $Y$ is connected.

Let $C$ be a log canonical center of $(X,B)$.
Let $\mu\colon (X',B_{X'})\to (X,B)$ be a log resolution such that
$\mu^{-1}(C)$ has codimension one. Let $\Sigma$ be the part of 
$B_{X'}^{=1}$ contained in $\mu^{-1}(C)$. We have $\mu(\Sigma)=C$.
Let $B'=B_{X'}-\Sigma$ and $N=N(B')=N(B_{X'})$. We have 
$$
-\Sigma-N\sim_\R K_{X'}+\Delta(B')
$$
The boundary $\Delta(B')$ supports $N^+$. By Remark~\ref{refor}, 
we obtain a surjection
$$
\Gamma(X',\cO_{X'}(-\Sigma+N^-))\to \Gamma(N^+,\cO_{N^+}(-\Sigma+N^-)). 
$$
We have $\Gamma(X',\cO_{X'}(-\Sigma+N^-))\subseteq \Gamma(X,\cI_C)=0$.
Therefore $\Gamma(X',\cO_{X'}(-\Sigma+N^-))=0$. We obtain
$
\Gamma(N^+,\cO_{N^+}(-\Sigma+N^-))=0.
$
Since 
$$
0=\Gamma(N^+,\cO_{N^+}(-\Sigma+N^-))\subseteq \Gamma(N^+,\cO_{N^+}(N^-))\ne 0,
$$ 
we infer $\Sigma\cap N^+\ne \emptyset$. This implies $C\cap Y\ne \emptyset$. 
\end{proof}

The next application is a corollary of~\cite[Theorem 4.4.]{Amb03}, if $H$ is $\Q$-ample.

\begin{thm}[Extension from lc centers]\label{elc}
Let $(X,B)$ be a proper log variety with log canonical singularities.
Let $L$ be a Cartier divisor on $X$ such that $H=L-(K_X+B)$ is a 
semiample $\Q$-divisor. Let $m_0\ge 1$, $D\in |m_0H|$, and denote
by $Z$ the union of lc centers of $(X,B)$ contained in $\Supp D$. Then
the restriction homomorphism
$$
\Gamma(X,\cO_X(L))\to \Gamma(Z,\cO_Z(L))
$$
is surjective.
\end{thm}

\begin{proof} By Lemma~\ref{ad}, there exists $\epsilon\in (0,1)\cap \Q$ 
such that $(X,B+\epsilon D)_{-\infty}=Z$. Let $m_1\ge 1$ such that the linear
system $|m_1H|$ has no base points. Let $D'\in |m_1H|$ be a general
member. Then $(X,B+\epsilon D+(\frac{1}{m_1}-\frac{\epsilon}{m_0m_1})D')_{-\infty}=Z$ and 
$$
L\sim_\Q  K_X+B+\epsilon D+(\frac{1}{m_1}-\frac{\epsilon}{m_0m_1})D'.
$$
By Theorem~\ref{8.2}, $\Gamma(X,\cO_X(L))\to \Gamma(Z,\cO_Z(L))$ is surjective.
\end{proof}

\begin{cor} Let $(X,B)$ be a proper log variety with log canonical 
singularities such that the linear system $|m_1(K_X+B)|$ has no base 
points for some $m_1\ge 1$. Let $m_0\ge 1$, $D\in |m_0(K_X+B)|$, and 
denote by $Z$ the union of lc centers of $(X,B)$ contained in $\Supp D$. 
Then
$$
\Gamma(X,\cO_X(mK_X+mB))\to \Gamma(Z,\cO_Z(mK_X+mB))
$$
is surjective for every $m\ge 2$ such that $mK_X+mB$ is Cartier.
\end{cor}

\begin{proof} Apply Theorem~\ref{elc} to $m(K_X+B)=K_X+B+(m-1)(K_X+B)$. 
\end{proof}


\section{Questions}


\begin{question}\label{inde} 
Let $(X,\Sigma)$ be a log smooth pair, with $X$ proper.
Denote $U=X\setminus \Sigma$. Is the restriction
$
H^q(X,\Omega^p_X(\log \Sigma))\to H^q(U,\Omega^p_U)
$ 
injective for $p+q>\dim X$?
\end{question}

\begin{exmp}
Let $P\in S$ be the germ of non-singular point, of dimension $d\ge 2$.
Let $\mu\colon X\to S$ be the blow-up at $P$, with exceptional locus
$E\simeq \bP^{d-1}$. Denote $U=X\setminus E$. The residue map 
$$
R^{d-1}\mu_*\cO_X(K_X+E)\to R^{d-1}\mu_*\cO_E(K_E)
$$
is an isomorphism, so $R^{d-1}\mu_*\cO_X(K_X+E)$ is a skyscraper sheaf
on $X$ centered at $P$. Since $\mu$ is an isomorphism on $U$,
$R^{d-1}(\mu|_U)_*\cO_E(K_E)=0$. Therefore the restriction homomorphism
$$
R^{d-1}\mu_*\cO_X(K_X+E)\to R^{d-1}(\mu|_U)_*\cO_U(K_U)
$$
is not injective.
\end{exmp}

\begin{question} Let $(X,\Sigma)$ be a log smooth pair. 
Denote $U=X\setminus \Sigma$.
Let $\pi\colon X\to S$ be a proper morphism, 
let $\pi|_U\colon U\to S$ be its restriction to $U$. 
Suppose that $\pi(C)=\pi(X)$ for every strata $C$ of $(X,\Sigma)$.
Is the restriction
$R^q\pi_*\cO_X(K_X+\Sigma)\to R^q(\pi|_U)_*\cO_U(K_U)$ injective
for all $q$?
\end{question}

\begin{question}
Let $(X,B)$ be a proper log variety with log canonical singularities. 
Let $U$ be the totally canonical locus of $(X,B)$. 
Let $L$ be a Cartier divisor on $X$ such that $L\sim_\R K_X+B$.
Is the restriction 
$
H^q(X,\cO_X(L))\to H^q(U,\cO_U(L|_U))
$
injective for all $q$?
\end{question}

\begin{question} Let $(X,B)$ be a proper log variety. Suppose the locus 
of non-log canonical singularities $Y=(X,B)_{-\infty}$ is non-empty.
Let $L$ be a Cartier divisor on $X$ such that $L\sim_\R K_X+B$.
Does the long exact sequence induced in cohomology by the short exact 
sequence 
$
0\to \cI_Y(L) \to \cO_X(L) \to \cO_Y(L)\to 0
$
break up into short exact sequences
$$
0\to H^q(X,\cI_Y(L))\to H^q(X,\cO_X(L))\to H^q(Y,\cO_Y(L))\to 0
\ (q\ge 0)?
$$
\end{question}


\end{document}